\documentclass[12pt]{amsart}
\usepackage[
    margin=1in
]{geometry}

\usepackage{amsmath,amssymb,amsxtra,amsthm,hyperref}
\usepackage{tikz}
\usepackage{multicol}
\usetikzlibrary{calc,decorations.markings}
\usepackage{enumitem} 

\newcommand{\bh}[1] {\mathcal{B}(\mathcal{#1})}

\newcommand{\lspan} {\operatorname{span}}

\newcommand{\condref}[1] {\hyperref[cond.#1]{(#1)}}

\newcommand{\cB}{\mathcal{B}}

\newcommand{\cH}{\mathcal{H}}

\newcommand{\cK}{\mathcal{K}}

\newcommand{\ds}{\displaystyle}

\newcommand{\bF}{\mathbb{F}}

\newtheorem{theorem}{Theorem}[section]

\newtheorem{corollary}[theorem]{Corollary}
\newtheorem{proposition}[theorem]{Proposition}   
 
\newtheorem{lemma}[theorem]{Lemma}

\theoremstyle{definition}
\newtheorem{definition}[theorem]{Definition}

\newtheorem{example}[theorem]{Example}
\newtheorem{remark}[theorem]{Remark}

\numberwithin{equation}{section}

\begin{document}

\title{Poisson transforms on right-angled Artin monoids}

\date{\today}


\begin{abstract}  We introduce the notion of the weak Brehmer's condition and prove that the Cauchy transform for a representation of a right-angled Artin monoid is bounded under such conditions. As a result,
we obtain the Poisson transform and $*$-regular dilation for a family of operators that satisfies the weak Brehmer's condition and the property (P). This generalizes Popescu's notion of Cauchy and Poisson transforms for commuting families of row contractions. 
\end{abstract}

\author{Boyu Li}
\address{Department of Mathematical Sciences, New Mexico State University, Las Cruces, New Mexico, 88003, USA}
\email{boyuli@nmsu.edu}
\date{\today}

\thanks{The author is supported by an NSF grant (DMS-2350543).}

\maketitle

\section{Introduction}

In \cite{Popescu1999}, Popescu introduced the Cauchy and Poisson transforms on a certain class of operators on Hilbert spaces, generalizing the earlier notion of such transforms on a single contractive operator and commuting operators \cite{CV1993, Vasilescu1992}. These transforms became an essential tool in the study of noncommutative geometry and multivariate operator theory (see \cite{AP2001, AP2000,MS2009, Popescu_2001_ProcEdin, Popescu2001, Popescu2003, Popescu_variety, Popescu_variety2, Popescu2010, BTS2006,STV2018}; and also \cite{Skalski2009, SZ2010} for generalized Poisson transform on product systems and higher rank graphs). 

The family of operators considered in \cite{Popescu1999} is closely related to representations of right-angled Artin monoids. A right-angled Artin monoid is generated by a set of generators that are either free or commuting, as dictated by the connectivity of an underlying graph. A representation of a right-angled Artin monoid is thus determined by a family of operators, where some of them are commuting. The family of operators considered in \cite{Popescu1999} consists of families of row contractions that commute with each other. This can be viewed as a representation of a right-angled Artin monoid, whose underlying graph is a complete multipartite graph. Therefore, it is natural to consider extending Popescu's result in \cite{Popescu1999} to contractive representations on general right-angled Artin monoids. 

A key consequence of Popescu's Poisson kernel is the $\ast$-regular dilation. The notion of $\ast$-regular dilation originated from Brehmer's study \cite{Brehmer1961} on dilations of commuting contractions. Recently, the theory of $\ast$-regular dilation on general semigroups was further developed in \cite{BLi2019}. It was shown in \cite{BLi2019} that having $\ast$-regular dilation is equivalent to a generalized Brehmer's condition. 
However, the proofs for $\ast$-regular dilation in \cite{Li2017, BLi2019} differ vastly from that in \cite{Popescu1999}. In \cite{Popescu1999}, it is shown that there is a Cauchy transform associated with commuting families of row contractions. Assuming the property (P), one can further define the Poisson transform, which is a unital completely positive map on a certain C*-algebra. The Stinespring dilation of this map gives the desired $\ast$-regular dilation.  

To extend Popescu's techniques to general right-angled Artin monoids, we first introduce the notion of weak Brehmer's condition on general right-angled Artin monoids (Definition~\ref{def.weak.brehmer}) that replaces the row-contraction conditions in the special case considered in \cite{Popescu1999}. Compared with Brehmer's condition in \cite{BLi2019}, the weak Brehmer's condition merely requires Brehmer's condition to hold for each subgraph corresponding to a connected component of the complement graph. 
Under this condition, we define and study the associated Cauchy transform. 
While it is natural to write down the formula for the Cauchy transform, it is unclear why such a map is bounded. 
We establish that the Cauchy transform is indeed bounded (Theorem~\ref{thm.Cauchy.bounded}), following two technical lemmas on the combinatorial property of the monoid (Lemma~\ref{lm.key.estimate} and \ref{lm.clique}). 
As a result, one can finally define the Poisson transform assuming the weak Brehmer's condition and property (P) (Theorem \ref{thm.poisson.main}). This gives another equivalence condition for $*$-regular dilation on right-angled Artin monoid (Corollary~\ref{cor.star.regular}). 
These generalized Cauchy and Poisson transforms open the door to further studies on the multivariate operator theory on right-angled Artin monoids.


\section{Right-angled Artin monoids and right LCM monoids}

Right-angled Artin monoid arises naturally in many different contexts. It is closely related to the class of Artin monoids and associated Coexter groups. One can also view a right-angled Artin monoid as a graph product of $\mathbb{N}$. As monoids, they are quasi-lattice ordered \cite{CrispLaca2002} and thus also right LCM. As an important class of quasi-lattice ordered semigroups, their semigroup C*-algebras and representations are well studied \cite{CrispLaca2002, CrispLaca2007, ELR2016, FS2022, Ivanov2010, LOS2021}.

Fix a finite simple graph $\Gamma=(V,E)$ where $V=\{1,2,\dots, n\}$ is the set of $n$ vertices and $E$ is the set of edges. The right-angled Artin monoid $A_\Gamma^+$ is defined by
\[A_\Gamma^+=\langle e_1,\dots,e_n: e_ie_j=e_je_i\ \text{ if } ij \in E\rangle.\]
In other words, each vertex $i$ of the graph $\Gamma$ corresponds to a generator $e_i$, and each edge $ij\in E$ of the graph corresponds to a commuting relation between $e_i$ and $e_j$. Note that when $ij\notin E$, the generators $e_i$ and $e_j$ are free. We use $1\in A_\Gamma^+$ to denote the identity element of the monoid. 
In an extreme case where the graph $\Gamma$ is a complete graph on $n$ vertices, the monoid $A_\Gamma^+$ is simply the free abelian monoid $\mathbb{N}^n$. On the other extreme when the graph $\Gamma$ contains no edge, $A_\Gamma^+$ is the free semigroup $F_n^+$ on $n$ generators. Therefore, one can treat a general Artin monoid as an interpolation between these two extreme cases. 

The restriction of a graph $\Gamma$ on $V'\subset V$ is defined by the graph $\Gamma|_{V'}=(V', E')$ where $E'=\{ij\in E: i,j\in V'\}$. In this case, $A_{\Gamma|_{V'}}^+$ is a submonoid of $A_\Gamma^+$.

For a graph $\Gamma=(V,E)$, we define its complement graph $\Gamma^c=(V,E^c)$ to be a graph on the same vertex set $V$ as $\Gamma$ and $E^c=\{ij: ij\notin E\}$. In other words, it consists of all those edges not in $\Gamma$. For example, if $\Gamma$ is the complete graph, then $\Gamma^c$ is the graph with no edges. Suppose $\{V_i: 1\leq i\leq k\}$ are the vertex set of all the connected components of $\Gamma^c$. Then for any vertex $j\in V_i$ and $j'\in V_{i'}$ ($i\neq i'$), $jj'\notin E^c$ since they belong to different connected components in $\Gamma^c$. This implies that $jj'\in E$ in the graph $\Gamma$. Now let $\Gamma_i=\Gamma|_{V_i}$ be the restriction of $\Gamma$ on the vertex set $V_i$. 
The generators of the submonoid $A_{\Gamma_i}^+$ commutes with those of $A_{\Gamma_{i'}}^+$ when $i\neq i'$. As a result, the right-angled artin monoid $A_\Gamma^+$ decomposes as a direct product 
\[A_\Gamma^+ = \prod_{i=1}^k A_{\Gamma_i}^+.\]
One may recall that these connected components of $\Gamma^c$ arise naturally in the study of the boundary quotient C*-algebra \cite{CrispLaca2007}. 

\begin{example}\label{ex.completeK} Let $n_1, n_2, \dots, n_k$ be positive integers. A complete $k$-partite graph $K_{n_1, n_2, \dots, n_k}$ is defined as a graph on $n=\sum_{i=1}^k n_i$ vertices, labeled by $\{v_{i,j}: 1\leq i\leq k, 1\leq j\leq n_i\}$. Its edge set consists of all  $v_{i_1, j_1}v_{i_2,j_2}$ where $i_1\neq i_2$. For example, the following diagram shows the complete $3$-partite graph $K_{2,2,1}$.
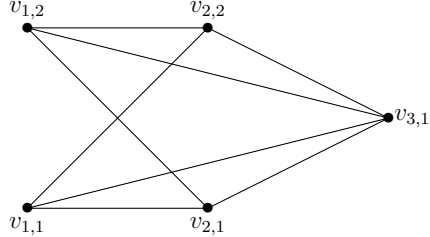
\begin{figure}[h]
    \centering

    \begin{tikzpicture}[scale=0.8, every node/.style={scale=0.8}]

    \draw[-] (0,0) -- (3,0);
    \draw[-] (0,0) -- (3,3);
    \draw[-] (0,0) -- (6,1.5);
    \draw[-] (0,3) -- (3,0);
    \draw[-] (0,3) -- (3,3);
    \draw[-] (0,3) -- (6,1.5);
    \draw[-] (3,0) -- (6,1.5);
    \draw[-] (3,3) -- (6,1.5);

    \node at (0,0){$\bullet$};
    \node at (0,3){$\bullet$};
    \node at (3,0){$\bullet$};
    \node at (3,3){$\bullet$};
    \node at (6,1.5){$\bullet$};
    
    \node at (0,-0.3){$v_{1,1}$};
    \node at (0,3.3){$v_{1,2}$};
    \node at (3,-0.3){$v_{2,1}$};
    \node at (3,3.3){$v_{2,2}$};
    \node at (6.4,1.5){$v_{3,1}$};

    \end{tikzpicture}
\end{figure}

It is easy to check that the complement graph of $K_{2,2,1}$ contains three connected components, and $A_\Gamma^+$ can be decomposed as 
\[A_\Gamma^+ = \bF_2^+\times \bF_2^+ \times \mathbb{N}.\]
\end{example}

A typical element $x\neq 1 \in A_\Gamma^+$ has the form $x=x_1 x_2 \cdots x_m$, where each $x_j=e_{i_j}^{a_j}$ for some $i_j\in\{1,2,\dots,n\}$ and $a_i\geq 1$. 
Since the only relations for $A_\Gamma^+$ are the commutation between some pairs of $e_i$ and $e_j$, one can define the \emph{norm} of $x$ to be $|x|=\sum_{j=1}^m a_j$. It is clear that $|xy|=|x|+|y|$, and by convention, we set $|1|=0$.
The $x_i$'s are often called syllables of $x$. We define $I(x_j)=i_j$ to denote the corresponding vertex for the $j$-th syllable. 
An expression for $x\in A_\Gamma^+$ is not unique in general. If ${i_j}$ and ${i_{j+1}}$ are two adjacent vertices, $x_{j}=e_{i_j}^{a_j}$ commutes with $x_{{j+1}}=e_{i_{j+1}}^{a_{j+1}}$ so that one can switch the positions of $x_j$ with $x_{j+1}$ to write  $x=x_1\cdots x_{i-1} x_{i+1} x_i x_{i+2} \cdots x_n$. This is called a \emph{shuffle} of the expression. If ${i_j}={i_{j+1}}$, then one can merge the two syllables as one by $x_jx_{j+1}=e_{i_j}^{a_{i_j}+a_{i_{j+1}}}$. This is called an \emph{amalgamation} of the expression. An expression is called \emph{reduced} if it cannot be shuffled to an expression that allows an amalgamation. A result of Green \cite{Green1990} states that reduced expressions are shuffle equivalent. As a result, the total number of syllables used in a reduced expression is a fixed number for each element in $A_\Gamma^+$. This is defined as the length of the element $x$, denoted by $\ell(x)$. We set $\ell(1)=0$ by convention.  
One should note that the difference between $\ell(x)$ and $|x|$ is that the first counts the number of syllables while the latter counts the number of generators. 

A vertex $i$ is called an \emph{initial vertex} of $x$ if there is a reduced expression of $x=x_1x_2\cdots x_m$ where $x_1=e_i^{a_1}$, $a_1\geq 1$. Here, $x_1=e_i^{a_i}$ is called an \emph{initial syllable}. In other words, a syllable is an initial syllable if we can shuffle it to the front. Similarly, if there is a reduced expression of $x=x_1x_2\cdots x_m$ where $x_m=e_i^{a_m}$, $a_m\geq 1$, then $x_m$ is called a \emph{final syllable} and its corresponding vertex $I(x_m)$ is called a \emph{final vertex}. 



Right-angled Artin monoids belong to the class of quasilattice ordered semigroups \cite{CrispLaca2002}, and more generally, right LCM monoids. Here, we briefly review the basics of right LCM monoids. 

A semigroup is a set equipped with an associative binary map called the multiplication map, and a monoid is a semigroup with an identity element.
Given a monoid $P$, we say $x\leq y$ if $y=xp$ for some $p\in P$. A monoid $P$ is called \emph{right LCM} if for any $p,q\in P$, either $pP\cap qP=\emptyset$ or $pP\cap qP=r P$ for some $r\in P$. When the latter happens, $r=pp'=qq'$ can be viewed as a least common multiple in $P$, and we denote it by $p\vee q=r$. The choice of $r$ is unique when $P$ contains no invertible elements other than the identity, which is a property that $A_\Gamma^+$ satisfies. 
When $pP\cap qP=\emptyset$, we often write $p\vee q=\infty$. For any finite subset $F\subset P$, one can repeatedly apply the right LCM condition to see that either $\bigcap_{p\in F} pP=\emptyset$ or $\bigcap_{p\in F} pP=rP$ for some $r\in P$. Again, we write $\vee F=\infty$ or $\vee F=r$ to denote these two cases respectively. 

The following Lemma gives a necessary condition for $\vee F\neq\infty$ where $F\subset A_\Gamma^+$ is finite.

\begin{lemma}\label{lm.upper.bound} Let $F\subset A_\Gamma^+$ be a finite subset and suppose $\vee F\neq\infty$. Suppose $I$ is an initial vertex for some $x\in F$. Then for each $y\in F$, either $I$ is an initial vertex of $y$ or $I$ is adjacent to every vertex of syllables of $y$. 
\end{lemma}

\begin{proof} Suppose $I$ is not an initial vertex for some $y\in F$. Let $r=\vee F$ and we have $r=xx'=yy'$ for some $x', y'\in A_\Gamma^+$. Since $I$ is an initial vertex of $x$, it is also an initial vertex of $r$. 

Now $r=yy'$ but $I$ is not an initial vertex of $y$. Therefore, the initial syllable for $I$ in $r$ must come from $y'$. In order to shuffle $e_I$ from $y'$ to the front of $r=yy'$, $e_I$ must commute with every syllable of $y$ and therefore $I$ is adjacent to every vertex of $y$. 
\end{proof}

Finally, recall that a subset of vertices $W$ in a graph $\Gamma$ is called a clique if all the vertices in $W$ are adjacent. A finite subset $U\subset F=\{e_1,\dots,e_n\}$ has $\vee U<\infty$ if and only if $U$ corresponds to a clique in $\Gamma$, in which case, $\vee U=\prod_{e_i\in U} e_i$.

\section{Brehmer conditions and regular dilation}

We now turn our attention to representations of right-angled Artin monoids $A_\Gamma^+$. A representation $T$ of the monoid $A_\Gamma^+$ is uniquely determined by its values on the set of generators, denoted by $T_j=T(e_j)$. The family of operators $\{T_j\}$ satisfies $T_i, T_j$ commute whenever $ij$ is an edge in $\Gamma$. Conversely, any such family defines a representation of the monoid $A_\Gamma^+$. Therefore, one can also approach the representations of $A_\Gamma^+$ from a multivariate operator perspective. We call a family of contractive operators $\{T_j\}_{j=1}^n$ a $\Gamma$-family if $T_i$ commutes with $T_j$ whenever $ij$ is an edge in $\Gamma$. We call it isometric if each $T_j$ is an isometry. 

We call an isometric $\Gamma$-family $\{V_j\}$ Nica-covariant if 
\[V_i^* V_j=
\begin{cases}
    V_jV_i^*, &\text{ if } ij\in E, \\
    0, &\text{ if }ij\notin E.
\end{cases}
\] 
We note that the Nica-covariant condition was first defined on general quasilattice ordered semigroups in \cite{Nica1992}. Our definition of Nica-covariance coincides with Nica's original definition due to \cite[Theorem 24]{CrispLaca2002}. 

The archetypal isometric Nica-covariant representation is the left-regular representation, denoted by $\lambda: A_\Gamma^+\to \cB(\ell^2(A_\Gamma^+))$. Here, fix an orthonormal basis $\{e_q: q\in A_\Gamma^+\}$, each $\lambda_p$ acts as a left shift operator by $\lambda_p e_q=e_{pq}$. One can verify that $\lambda$ is Nica-covariant. The C*-algebra generated by the left-regular representation $\lambda$ is called the \emph{reduced semigroup C*-algebra} of the monoid $A_\Gamma^+$, denoted by $C^*_\lambda(A_\Gamma^+)=C^*(\{\lambda_p: p\in A_\Gamma^+\})\subset\cB(\ell^2(A_\Gamma^+))$. From the Nica-covariance condition, 
\[\lambda_p^* \lambda_q = \begin{cases}
    \lambda_{p^{-1}r} \lambda_{q^{-1}r}^*, &\text{if } p\vee q=r\neq\infty,\\
    0, &\text{otherwise}. 
    \end{cases}\]
Therefore, it is easy to check that $C_\lambda^*(A_\Gamma^+)=\overline{\operatorname{span}}\{\lambda_p \lambda_q^*: p,q\in A_\Gamma^+\}$. 

Given a contractive representation $T:A_\Gamma^+\to\bh{H}$, we say $V:A_\Gamma^+\to\bh{K}$ is a dilation of a contractive representation $T$ if for any $p\in A_\Gamma^+$, 
\[P_\cH V_p |_\cH = T_p.\]
This dilation is called minimal if $\cK=\bigvee_{p\in A_\Gamma^+} V_p\cH$. 

We say $V$ is a $*$-regular dilation of  $T$ if $\cH\subset\cK$ and for any $p,q\in A_\Gamma^+$,
\[P_\cH V_p^* V_q|_\cH = \begin{cases}
    T_{p^{-1}r} T_{q^{-1}r}^*, &\text{if } p\vee q=r\neq\infty,\\
    0, &\text{otherwise}. 
    \end{cases}\]
Notice that any $*$-regular dilation is also a dilation by setting $p=1$. 

The notion of $*$-regular dilation was first introduced by Brehmer \cite{Brehmer1961} for $n$ commuting contractions. It is observed in \cite[Theorem 3.9]{BLi2019} (also \cite{Li2017} for $A_\Gamma^+$) that having $*$-regular dilation is equivalent to having a minimal isometric Nica-covariant dilation, which is also equivalent to a generalized Brehmer's condition.

\begin{theorem}\label{thm.regular.LCM} Let $\Gamma$ be a graph on $n$ vertices and $T:A_\Gamma^+\to\bh{H}$ be a unital representation of $A_\Gamma^+$. Then the following are equivalent. 
\begin{enumerate}[label=\textup{(\arabic*)}]
    \item $T$ has $*$-regular dilation.
    \item $T$ has an isometric Nica-covariant dilation.
\end{enumerate}
\begin{enumerate}[label=\textup{(B)}]
    \item\label{cond.Brehmer.Artin} For each clique $W$ in $\Gamma$
\[
        Z({N_\Gamma(W))} = \sum_{U\subseteq N_\Gamma(W)} (-1)^{|U|} T_{\vee U} T_{\vee U}^* \geq 0.
\]
    Here, 
    \[N_\Gamma(W)=\{e_i \in\Gamma: ij\in E \text{ for all }j\in W\},\]
    and $T_{\vee U}=0$ if $\vee U=\infty$.
\end{enumerate}
\end{theorem}

The relation between $\ast$-regular dilation and Brehmer's Condition was further developed in \cite{BLi2019}. It is shown that for a contractive representation $T$ on a general right LCM semigroup $P$, having $\ast$-regular dilation is equivalent to the generalized Brehmer's Condition:
\begin{enumerate}[label=\textup{(B')}]
    \item\label{cond.Brehmer} For any finite $F\subset P$, 
    \[
        Z(F) = \sum_{U\subseteq F} (-1)^{|U|} T_{\vee U} T_{\vee U}^* \geq 0.
    \]
\end{enumerate}
In the case of the Artin monoid, Condition~\ref{cond.Brehmer.Artin} and \ref{cond.Brehmer} are equivalent.

\begin{remark} We would like to point out that 
the original Condition~\ref{cond.Brehmer.Artin} in \cite{Li2017} requires $Z(F)\geq 0$ for all finite $F\subset \{e_1,\dots,e_n\}$. 
However, the proof only requires checking finite subset $F$ of the form $\{e_i: i\in N_\Gamma(W)\}$ where $W$ is a clique in $\Gamma$. One can see this by following the proof outlined in \cite[Theorem 4.13 and Proposition 6.14]{Li2017}. 

We also note that if $W$ is a maximal clique, meaning it is not contained in any larger clique. Then $N_\Gamma(W)=\emptyset$, in which case, $Z(N_\Gamma(W))=I$ is trivially positive. Therefore, the Condition~\ref{cond.Brehmer.Artin} can be further reduced to checking non-maximal cliques $W$ in $\Gamma$. 

For example, when $\Gamma$ is the empty graph on $n$ vertices, the cliques of $\Gamma$ consist of $\emptyset$ and all singleton sets $\{i\}$. The only condition required would be $Z(N_\Gamma(\emptyset))\geq 0$, which is the row contraction condition in this case. 
\end{remark}

In \cite{Popescu1999}, Popescu showed that certain families of contractions have $*$-regular dilation. 
Let $\Gamma$ be a complete $k$-multipartite graph $K_{n_1, \dots,n_k}$. A $\Gamma$-family is given by $\{T_{i,j}: 1\leq i\leq k, 1\leq j\leq n_i\}$, where $T_{i,j}$ commutes with $T_{i',j'}$ whenever $i\neq i'$. Popescu considerd such families $\{T_{i,j}\}$ where $\{T_{i,j}: 1\leq j\leq n_i\}$ is a row contraction for each $i$. It is shown in \cite{Popescu1999} that such $\{T_{i,j}\}$ has a $*$-regular dilation if it satisfies the property (P). Details of property (P) will be discussed in Section~\ref{sec.Poisson}.  

We note that without the property (P), merely requiring row contractiveness is weaker than Brehmer's Condition~\ref{cond.Brehmer.Artin}. For example, consider the case when $\Gamma$ is a complete graph $K_n$ on $n$ vertices. A $\Gamma$-family consists of $n$ commuting contractions $\{T_j: 1\leq j\leq n\}$. Popescu proved that $\{T_j\}$ has a $*$-regular dilation if it satisfies property (P). However, without the property (P), a family of $n$ commuting contraction may not even have isometric dilation. 

Therefore, it seems that Brehmer's Condition (Condition~\ref{cond.Brehmer.Artin}) is equivalent to the property (P) plus a weaker version of Condition~\ref{cond.Brehmer.Artin}. Attempts were made in \cite{Li2017} to explore their connection, but it was unclear how to replace the row contractive condition in a general right-angled Artin monoid. 

We now introduce the notion of weak Brehmer's condition, which replaces the row contraction condition. 

\begin{definition}\label{def.weak.brehmer}  We say a $\Gamma$ family $\{T_j\}_{j=1}^n$ satisfies the weak Brehmer's condition if for each subgraph $\Gamma_i$ of $\Gamma$ corresponding to a connected component of $\Gamma^c$, we have $\{T_j: j\in \Gamma_i\}$ satisfies Condition~\ref{cond.Brehmer.Artin} as a $\Gamma_i$-family. 
\end{definition}

\begin{example} Consider the case when $\Gamma=K_{n_1,n_2,\dots,n_k}$ with vertex set $V=\{v_{i,j}: 1\leq i\leq k, 1\leq j\leq n_i\}$ as considered in Example~\ref{ex.completeK}. For each $i$, the vertices $V_i=\{v_{i,j}: 1\leq j\leq n_i\}$ correspond to a connected component of $\Gamma^c$. The restriction $\Gamma_i=\Gamma|_{V_i}$ is an empty graph on $n_i$ vertices. Therefore, the weak Brehmer's condition requires that $\{T_{i,j}: 1\leq j\leq n_i\}$ forms a row contraction for each $i$. This is precisely the family of operators considered in \cite{Popescu1999}. 
\end{example}

\begin{example}\label{ex.toy.graph}
Let $\Gamma$ be the following graph. 
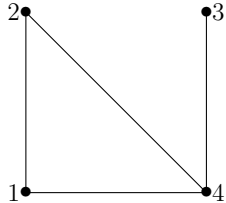
\begin{figure}[h]
    \centering

    \begin{tikzpicture}[scale=0.8, every node/.style={scale=0.8}]

	\draw[-] (0,0) -- (0,3);
    \draw[-] (0,3) -- (3,0);
    \draw[-] (3,3) -- (3,0);
	\draw[-] (0,0) -- (3,0);
    
    \node at (0,0){$\bullet$};
    \node at (0,3){$\bullet$};
    \node at (3,0){$\bullet$};
    \node at (3,3){$\bullet$};
    
    \node at (-0.2,0){$1$};
    \node at (-0.2,3){$2$};
    \node at (3.2,0){$4$};
    \node at (3.2,3){$3$};

    \end{tikzpicture}
\end{figure}

A $\Gamma$-family of $A_\Gamma^+$ consists of $4$ contractions $\{T_j\}_{j=1}^4$ where $T_1, T_2, T_4$ commute with each other and $T_3$ only commutes with $T_4$. The complement graph $\Gamma^c$ has two connected components: $\Gamma_1$ on vertices $\{1,2,3\}$ and $\Gamma_2$ on vertices $\{4\}$. The weak Brehmer's condition requires the following conditions:
\begin{align*}
    I-T_1T_1^*-T_2T_2^*-T_3T_3^*+T_{12}T_{12}^* \geq 0 & \text{ for }W=\emptyset\subset\Gamma_1, N_{\Gamma_1}(W)=\{1,2,3\},\\ 
    I-T_2T_2^* \geq 0 & \text{ for }W=\{1\}\subset\Gamma_1, N_{\Gamma_1}(W)=\{2\},\\ 
    I-T_1T_1^* \geq 0 & \text{ for }W=\{2\}\subset\Gamma_1, N_{\Gamma_1}(W)=\{1\},\\ 
    I-T_4T_4^* \geq 0 & \text{ for }W=\emptyset\subset\Gamma_2, N_{\Gamma_2
    }(W)=\{3\}.
\end{align*}
Since $I-T_iT_i^*\geq 0$ follows from the contractive condition, the weak Brehmer's condition only requires the $\Gamma$-family to satisfy the first inequality from the list above. 
\end{example}

\section{Cauchy transform}

Let $A_\Gamma^+$ be a right-angled Artin monoid associated with the graph $\Gamma$ on $n$ vertices. Let $\{T_j\}_{j=1}^n$ be a $\Gamma$-family that satisfies the weak Brehmer's condition. 
Fix a number $0\leq r<1$, we define the \emph{Cauchy transform} to be the map
\[C_{r,T} : \cH \to \ell^2(A_\Gamma^+)\otimes \cH,\]
by 
\[C_{r,T} h := \sum_{p\in A_\Gamma^+} e_p \otimes r^{|p|} T_p^* h.\]
Here, $\{e_p: p\in A_\Gamma^+\}$ is the canonical orthonormal basis for $\ell^2(A_\Gamma^+)$. 

When $A_\Gamma^+$ is a complete $k$-bipartite graph with $k$ independent sets of sizes $n_1, n_2, \dots, n_k$, the monoid $A_\Gamma^+$ can be viewed as a direct product of free semigroups, 
\[A_\Gamma^+ = \prod_{i=1}^k \bF_{n_i}^+.\]
In particular, the space $\ell^2(A_\Gamma^+)$ is isomorphic to $\bigotimes_{i=1}^k F^2(H_{n_i})$, where $H_{n_i}\cong \mathbb{C}^{n_i}$ and $F^2(H_{n_i})\cong \ell^2(\bF_{n_i}^+)$ is the full Fock space. In this case, the map $C_{r,T}$ is precisely the Cauchy transform defined in \cite{Popescu1999}. 

It is unclear a priori why the map $C_{r,T}$ is bounded, and showing $C_{r,T}$ is bounded for all $0\leq r<1$ is the main goal of this section. In particular, we need the convergence when $r$ is close to $1$ since we are going to take $r\to 1$ eventually. 

We first note that
\[\|C_{r,T} h\|^2 = \sum_{p\in A_\Gamma^+} r^{2|p|} \|T_p^* h\|^2 =\sum_{m=0}^\infty \left(\sum_{|p|=m} \|T_p^* h\|^2\right) r^{2m}\]
One naive estimate is to use $\|T_p\|\leq 1$ and $\|T^*_ph\|^2 \leq \|h\|^2$ to get:
\[\sum_{|p|=m} \|T_p^* h\|^2 \leq \left|\left\{p\in A_\Gamma^+: |p|=m\right\}\right| \cdot \|h\|^2.\]
However, this estimate is too loose. For example, in the case when $\Gamma$ is an empty graph on $2$ vertices, $\left|\left\{p\in A_\Gamma^+: |p|=m\right\}\right|=2^m$, and 
above estimate gives
\[\|C_{r,T} h\|^2 \leq \sum_{m=0}^\infty 2^m r^{2m} \|h\|^2.\]
The right-hand side converges only when $0 < r^2 < \frac{1}{2}$. 

Therefore, one needs a tighter estimate on $\sum_{|p|=m} \|T_p^* h\|^2$. In \cite{Popescu1999}, the proof that $C_{r,T}$ is bounded relies heavily on the fact that the operators $\{T_j\}$ consist of commuting row contractions. However, we need a new technique using the weak Brehmer's condition.

We start with the following technical lemma.

\begin{lemma}\label{lm.key.estimate} Let $P$ be a right LCM monoid. Let $T:P\to\bh{H}$ be a contractive representation that satisfies Condition~\ref{cond.Brehmer}. Let $F\subset P$ be a finite set. Then for any $h\in \cH$,
\begin{equation}\tag{$\lozenge$}\label{eq.diamond}
    \sum_{p\in F} \|T_p^* h\|^2  \leq c_F \|h\|^2.
\end{equation}

where $c_F=\max\{\ell: U\subset F, |U|=\ell, \vee U\neq \infty\}$. In other words, $c_F$ is the largest size of subsets of $F$ that have a common upper bound in $P$. 
\end{lemma}

\begin{proof}
Let $F_k=\{\vee U: U\subset F, |U|=k\}$. We note that $F_1=F$ and $F_k=\emptyset$ for all $k>c_F$ by the definition of $c_F$. We note that $F_k$ may contain multiple copies of an element because different $U$ may generate the same element $\vee U$. This will not affect any of the computations. 

By Condition~\ref{cond.Brehmer}, $Z(F_k)\geq 0$ for all $1\leq k \leq c_F$, where
\[Z(F_k)=\sum_{U\subset F_k} (-1)^{|U|} T_{\vee U} T_{\vee U}^*.\]

We now claim that
\begin{equation*}
    c_F I - \sum_{p\in F} T_p T_p^* = \sum_{k=1}^{c_F} Z(F_k).
\end{equation*}
The desired result would follow easily from this claim since the right-hand side is positive and
\[\sum_{p\in F} \|T_p^* h\|^2 = \left\langle \left(\sum_{p\in F} T_p T_p^*\right) h, h\right\rangle \leq c_F \|h\|^2.\]

To prove this claim, we first notice that $\bigvee_{i}(\vee U_i)=\vee(\cup_i U_i)$. Therefore, all the terms in $Z(F_k)$ have the form $T_{\vee U} T_{\vee U}^*$ for some $U\subset F$. For each $\emptyset\neq U\subset F$ where $\vee U\neq\infty$, let us keep track of the term $T_{\vee U} T_{\vee U}^*$ in each of $Z(F_k)$, $k\leq |U|$. Each term for $\vee U$ in $Z(F_k)$ comes from $\vee U=\vee_{i=1}^m (\vee U_i)$ for some collection of $m$ elements $\{\vee U_i\}_{i=1}^m$ from $F_k$. This requires that $U_i$ are distinct subsets, each with $k$ elements, and $\cup_{i=1}^m U_i=U$. 
Since each $U_i$ is a subset of $U$ and there is a total of ${|U|\choose k}$ subsets of $U$ of size $k$, the total number of subsets $m$ ranges from $1$ to ${|U|\choose k}$. Therefore, the coefficient of $T_{\vee U} T_{\vee U}^*$ in $Z(F_k)$ is given by
\[\sum_{m=1}^{|U|\choose k} (-1)^m n_{m,k}^U.\]
Here, 
\[n_{m,k}^U=\left|\left\{\{U_1,U_2,\dots,U_m\}: U_i\text { are distinct subsets of }U, |U_i|=k, \bigcup_{i=1}^m U_i=U\right\}\right|.\] In other words, $n_{m,k}^U$ is the number of ways of writing $U$ as a union of $m$ distinct subsets, where each subset contains $k$ elements. Sum this over all the $Z(F_k)$ where $k\leq |U|$, we obtain the coefficient of $T_{\vee U} T_{\vee U}^*$ in $\sum_{k=1}^{c_F} Z(F_k)$ is 
\[n_U=\sum_{k=1}^{|U|} \sum_{m=1}^{{|U|\choose k}} (-1)^m n_{m,k}^U\]

When $|U|=1$, $n_U=-1$ since $n_{m,k}^U=0$ except when $m=k=1$, $n_{m,k}^U=1$.

We now claim that $n_U=0$ for all $|U|>1$. Write $U=\{u_1,\dots, u_{|U|}\}$. For each subset $E\subset U$, define 
\[N_E=\left\{\{U_1,\dots,U_m\}:U_i\text { are distinct subsets of }U, |U_i|=k, E\cap \left(\bigcup_{i=1}^m U_i\right)=\emptyset\right\} \]
In other words, $N_E$ contains all those collections of subsets that miss all the elements in $E$. The number
\[n_{m,k}^U = N_\emptyset/\left(\bigcup_{i=1}^{|U|} N_{\{u_i\}}\right),\]
which is the number of $\{U_1,\dots,U_m\}$ such that each element $u_i$ is in some $U_j$. By the inclusion-exclusion principle,
\[n_{m,k}^U = \sum_{E\subset U} (-1)^{|E|} |N_E|.\]
Now for each $N_E$, the subsets $U_j$ are chosen from all those subsets that are disjoint with $E$. There is a total of ${|U|-|E|\choose k}$ such subsets of size $k$, and therefore
\[|N_E|={{|U|-|E|\choose k} \choose m}.\]
Therefore,
\[n_{m,k}^U = \sum_{j=0}^{|U|}  \sum_{|E|=j} (-1)^{|E|} |N_E| = \sum_{j=0}^{|U|} (-1)^j {|U| \choose j} {{|U|-j\choose k} \choose m}.\]
Now
\begin{align*}
    n_U &= \sum_{k=1}^{|U|} \sum_{m=1}^{{|U|\choose k}} (-1)^m n_{m,k}^U \\
    &= \sum_{k=1}^{|U|} \sum_{m=1}^{{|U|\choose k}}  \sum_{j=0}^{|U|} (-1)^m(-1)^j {|U| \choose j} {{|U|-j\choose k} \choose m}
\end{align*}

Notice that the summand is $0$ unless $|U|-j\geq k$ and ${|U|-j\choose k} \geq m$. We can change the order of summation to:
\begin{align*}
    n_U &= \sum_{j=0}^{|U|-1} (-1)^j {|U| \choose j} \cdot \left(\sum_{k=1}^{|U|-j}  \sum_{m=1}^{{|U|-j \choose k}} (-1)^m {{|U|-j\choose k} \choose m}\right).
\end{align*}
Use the property of binomial coefficients that $\sum_{m=1}^L (-1)^m {L\choose m} = -1$, we have,
\begin{align*}
    n_U &= \sum_{j=0}^{|U|-1} (-1)^j {|U| \choose j} \sum_{k=1}^{|U|-j} (-1) \\
    &= - \sum_{j=0}^{|U|-1} (-1)^j (|U|-j) {|U| \choose j}
\end{align*}
From here, we can use $f(x)=(x+1)^{|U|}=\sum_{j=0}^{|U|} (-1)^j {|U|\choose j} x^{|U|-j}$. Take $f'(x)$ and plug in $x=-1$ gives $n_U=0$ when $|U|>1$.  
\end{proof}

\begin{example}
Let $\Gamma$ be the graph considered in Example~\ref{ex.toy.graph}. 
Let $F=\{1,2,3,4\}\subset A_\Gamma^+$. Since $\vee F=\infty$ and $\vee\{1,2,4\}=124<\infty$, we have $c_F=3$. 
Lemma~\ref{lm.key.estimate} claims that if $\{T_j\}$ satisfies Condition~\ref{cond.Brehmer}, then,
\[\sum_{i=1}^4 T_i T_i^* \leq 3 I\]
To see this, let us follow the proof of Lemma~\ref{lm.key.estimate} and set
\begin{align*}
    F_1 &=\{1,2,3,4\}, \\
    F_2 &=\{\vee U: U\subset F, |U|=2, \vee U\neq\infty\}=\{12,14,24,34\}, \\
    F_3 &= \{\vee U: U\subset F, |U|=3, \vee U\neq\infty\}=\{124\}.
\end{align*}
The corresponding Brehmer's conditions for these sets are given by:
\begin{align*}
    Z(F_1) &= I-\sum_{i=1}^4 T_iT_i^* + T_{12}T_{12}^* + T_{14}T_{14}^*+T_{24}T_{24}^*+T_{34}T_{34}^*-T_{124}T_{124}^*\geq 0, \\
    Z(F_2) &= I-T_{12}T_{12}^* - T_{14}T_{14}^*-T_{24}T_{24}^*-T_{34}T_{34}^*+3T_{124}T_{124}^* - T_{124}T_{124}^* \geq 0, \\
    Z(F_3) &= I-T_{124}T_{124}^* \geq 0.
\end{align*}
Summing these three inequalities gives $3I-\sum_{i=1}^4 T_i T_i^* \geq 0$. 

To better illustrate the idea in the proof of Lemma~\ref{lm.key.estimate}, let us keep track of the term $T_{124} T_{124}^*$ (corresponding to $U=\{1,2,4\}$), we can see that it appear from
\begin{enumerate}
    \item $124=1\vee 2\vee 4$ in $Z(F_1)$, contributing a coefficient of $(-1)^3 n_{3,1}^U=-1$.
    \item $124=12\vee 14=12\vee 24=14\vee24$ in $Z(F_2)$, contributing a coefficient of $3\cdot (-1)^2 n_{2,2}^U=3$. 
    \item $124=12\vee 24\vee 14$ in $Z(F_2)$, contributing a coefficient of $(-1)^3 n_{3,2}^U=-1$.
    \item Finally, $124\in F_3$, so it appears in $Z(F_3)$ with a coefficient of $-n_{3,1}^U=-1$. 
\end{enumerate}
Summing all these together, the coefficient for $T_{124} T_{124}^*$ becomes $0$ in the sum $\sum_{k=1}^3 Z(F_k)$. 
\end{example}

Now to give a norm estimate of the Cauchy transform on $A_\Gamma^+$. We need to derive an upper bound for $\sum_{|p|=m} \|T_p^* h\|^2$. Lemma~\ref{lm.key.estimate} converts this problem into a combinatorial problem of finding the largest subset of $\{p: |p|=m\}$ that has an upper bound if $T$ satisfies Brehmer's condition (Condition~\ref{cond.Brehmer}) which is equivalent to Condition~\ref{cond.Brehmer.Artin}.

Let $\{T_j\}$ be a $\Gamma$-family that satisfies the weak Brehmer's condition. 
One caveat here is that the weak Brehmer's condition only ensures Condition~\ref{cond.Brehmer} on each subgraph $\Gamma_i$ corresponding to a connected component of $\Gamma^c$. The family may not satisfy Condition~\ref{cond.Brehmer} so that one cannot apply Leamm~\ref{lm.key.estimate} directly. Nevertheless, we can first apply Lemma~\ref{lm.key.estimate} on each component $\Gamma_i$, and then use the fact that $A_\Gamma^+=\prod A_{\Gamma_i}^+$ and each $\Gamma_i$-family commutes with each other.

Let $w=\omega(\Gamma)$ be the size of the largest clique in $\Gamma$. The number $w$ is often called the \emph{clique number} of the graph $\Gamma$.

\begin{lemma}\label{lm.clique}  Let $\{T_j\}$ be a $\Gamma$-family on $\bh{H}$ that satisfies the weak Brehmer's condition. Let $\omega$ be the clique number of $\Gamma$. 
Then for each $m\geq 1$ and $h\in\cH$, 
\[\sum_{|p|=m} \|T_p^* h\|^2 \leq  {w+m-1 \choose m}\|h\|^2\]
\end{lemma}

\begin{proof} Let us first consider the case when $\Gamma^c$ is connected, in which case $\{T_j\}$ satisfies Brehmer's Condition~\ref{cond.Brehmer.Artin} and thus Condition~\ref{cond.Brehmer}. By Lemma~\ref{lm.key.estimate}, it suffices to show  $c_{F_m}={w+m-1 \choose m}$ where $E_m=\{p\in A_\Gamma^+: |p|=m\}$. 

Without loss of generality, assume $\{1,2,\dots, \omega\}$ is a clique in $\Gamma$ with the largest size. As a result, the generators $\{e_1,\dots, e_\omega\}$ commutes. If we let 
\[U_m=\{p\in \langle e_1,\dots, e_\omega\rangle\subset A_\Gamma^+: |p|=m\} = \left\{e_1^{a_1}\cdots e_\omega^{a_\omega}: a_j\geq 0, \sum_{j=1}^\omega a_j=m\right\}.\]
Then it is clear that $\vee U=e_1^me_2^m\cdots e_\omega^m<\infty$ and $\ds|U_m|={w+m-1 \choose m}$ (corresponding to the number of non-negative integer solutions $\{a_j\}$ to $\sum_{j=1}^\omega a_j=m$). 

Therefore, we only need to show that $c_{E_m}\leq {w+m-1 \choose m}$. We do an induction on $m$ and $\omega$. The result is clearly true for $m=1$ or $\omega=1$. 

Let $U=\{p_1, \dots, p_c\}\subset E_m$ be a subset such that $\vee U\neq\infty$.  Without loss of generality, we assume the vertex $1\in \Gamma$ is an initial vertex for $p_1$.
Reorder $p_i$ such that $p_1,\dots, p_{c_1}$ have the vertex $1$ as an initial vertex, while $p_{c_1+1},\dots, p_c$ do not. Since the vertex $1$ is an initial vertex in all of $p_1,\dots, p_{c_1}$, we can remove one $e_1$ on the left from them and obtain a subset $U_1=\{e_1^{-1}p_1,\dots, e_1^{-1}p_{c_1}\}\subset E_{m-1}$, and $\vee U_1\neq\infty$. By the induction hypothesis, $\ds|U_1| \leq {\omega + m-2 \choose m-1}$. 

On the other hand, consider $U_2=\{p_{c_1+1}, \dots, p_c\}$. By Lemma~\ref{lm.upper.bound}, the vertex $1$ must be adjacent to all vertices for these $p_j$. Let $N_1=\{j\in \Gamma: 1j\in E\}$ be all the vertices in $\Gamma$ that are adjacent to $1$ and $\Gamma_1=\Gamma|_{N_1}$ be the subgraph restricted to all such vertices. We have that for $c_1+1\leq j\leq c$, $p_j\in A_{\Gamma_1}^+\subset A_\Gamma^+$. 

Now $U_2\subset E_m \subset A_{\Gamma_1}^+$ satisfies $\vee U_2\neq \infty$. The largest clique in the subgraph $\Gamma_1$ cannot exceed $\omega-1$, since any clique in $\Gamma_1$ plus their common neighbour vertex $1$ forms a clique of size at most $\omega$ in $\Gamma$. Therefore, by the induction hypothesis,
$\ds|U_2| \leq {(\omega-1) + m -1 \choose m}$.

Combining these together,
\[|U|=|U_1|+|U_2|\leq {\omega + m-2 \choose m-1}+{\omega + m -2 \choose m} = {\omega+m-1\choose m}.\]

Now for a general graph $\Gamma$, let $\{\Gamma_i\}_{i=1}^k$ be all the subgraphs of $\Gamma$ that correspond to connected components of $\Gamma^c$. Let $\omega_i=\omega(\Gamma_i)$. It is clear that $\omega=\sum_{i=1}^k \omega_i$ since cliques in $\Gamma_i$'s will join together as a clique in $\Gamma$. Each $\Gamma_i$'s complement graph $\Gamma_i^c$ is connected. Each $\{T_j: j\in \Gamma_i\}$ is a $\Gamma_i$-family that satisfies Brehmer's Condition~\ref{cond.Brehmer.Artin}. Therefore, by previous calculations, for each $m\geq1$ and $h\in \cH$,
\begin{equation}\label{eq.connected}
    \sum_{\substack{p\in A_{\Gamma_i}^+ \\|p|=m}} \|T_p^*h\|^2 \leq {\omega_i + m-1 \choose m} \|h\|^2
\end{equation}
Now $A_\Gamma^+=\prod_{i=1}^k A_{\Gamma_i}^+$. Thus, an element $p\in A_\Gamma^+$ decomposes as a product $p=\prod_{i=1}^k p_i$ with $p_i\in A_{\Gamma_i}^+$. 
\[\sum_{\substack{p\in A_{\Gamma}^+ \\ |p|=m}} \|T_p^*h\|^2 = \sum_{\substack{p_i\in A_{\Gamma_i}^+ \\ |p_1|+\dots+|p_k|=m}} \|T_{p_1}^*\cdots T_{p_k}^*h\|^2\]
Repeatedly apply Equation~\eqref{eq.connected}, we have that
\[\sum_{\substack{p_i\in A_{\Gamma_i}^+ \\ |p_1|+\dots+|p_k|=m}} \|T_{p_1}^*\cdots T_{p_k}^*h\|^2 \leq \sum_{m_1+\dots+m_k=m} {\omega_i + m_i - 1 \choose m_i} \|h\|^2.\]
Recall the negative binomial expansion 
\[\frac{1}{(1-z)^{\omega_i}}=\sum_{m_i\geq 0} {\omega_i + m_i - 1 \choose m_i} z^{m_i}.\]
By considering the coefficient of $z^m=z^{m_1+\dots+m_k}$ in $\prod \frac{1}{(1-z)^{\omega_i}}=\frac{1}{(1-z)^\omega}$, we have
\[\sum_{m_1+\dots+m_k=m} {\omega_i + m_i - 1 \choose m_i}={\omega+m-1\choose m}.\]
This proves the desired inequality. 
\end{proof}

As a result, we can now show that the Cauchy transform is indeed a bounded map. 

\begin{theorem}\label{thm.Cauchy.bounded} 
Let $\{T_j\}$ be a $\Gamma$-family on $\bh{H}$ that satisfies the weak Brehmer's condition.
Then $C_{r,T}$ is bounded for all $0\leq r<1$. Moreover, let $\omega=\omega(\Gamma)$ be the clique number of $\Gamma$. For each $h\in \cH$, we have
\[\|C_{r,T} h\|^2 \leq \frac{1}{(1-r^2)^{\omega}} \|h\|^2.\]
\end{theorem}
\begin{proof}
    By Lemma~\ref{lm.key.estimate} and \ref{lm.clique}, we have the estimate:
    \[\sum_{|p|=m} \|T_p^* h\|^2 \leq {\omega+m-1 \choose m} \|h\|^2.\]
Therefore,
\[\|C_{r,T} h\|^2 = \sum_{m=0}^\infty \left(\sum_{|p|=m} \|T_p^* h\|^2\right) r^{2m} \leq \sum_{m=0}^\infty {\omega+m-1 \choose m} r^{2m} \|h\|^2 = \frac{\|h\|^2}{(1-r^2)^\omega}. \qedhere\]
\end{proof}

In the case when $\Gamma=K_{n_1,\dots,n_k}$ is a complete $k$-partite graph, the largest clique in $\Gamma$ has size $\omega=k$. The estimate $\ds\|C_{r,T}h\|^2 \leq \frac{\|h\|^2}{(1-r^2)^k}$ is the same estimated as obtained in \cite{Popescu1999}.  

Recall that $\lambda:A_\Gamma^+\to\cB(\ell^2(A_\Gamma^+))$ is the left regular representation of the monoid $A_\Gamma^+$. The following Proposition is an analogue of \cite[Theorem 2.1]{Popescu1999}. 

\begin{proposition}\label{prop.Cauchy} Let $\{T_j\}$ be a $\Gamma$-family on $\bh{H}$ that satisfies the weak Brehmer's condition. Let $C_{r,T}: \cH\to \ell^2(A_\Gamma^+)\otimes\cH$ be the Cauchy transform for $0\leq r<1$. Then for any $p\in A_\Gamma^+$ and $h,k\in \cH$, 
\[\langle r^{|p|} T_p h ,k\rangle = \langle (\lambda_p\otimes I_\cH)(e_1\otimes h), C_{r,T} k\rangle.\]
\end{proposition}

\begin{proof}
    Starting from the right-hand side,
    \begin{align*}
         \langle (\lambda_p\otimes I_\cH)(e_1\otimes h), C_{r,T} k\rangle &= \left\langle e_p\otimes h, \sum_{q\in A_\Gamma^+} e_q \otimes r^{|q|} T_q^*h\right\rangle \\
         &= \left\langle e_p\otimes h, e_p \otimes r^{|p|} T_p^*h\right\rangle \\
         &= \langle r^{|p|} T_p h ,k\rangle. \qedhere
    \end{align*}
\end{proof}


\section{Poisson transform}\label{sec.Poisson}


The goal of this section is to define the Poisson transform for a $\Gamma$-family $\{T_j\}$ with weak Brehmer's condition and the property (P). Using the properties of this Poisson transform, we shall prove that such a family has a $*$-regular dilation. Let us first rephrase Popescu's property (P) in terms of $\Gamma$-families of right-angled Artin monoid (see also \cite{Li2017}). 
\begin{definition} Let $\{T_j\}$ be a $\Gamma$-family. Let $F=\{e_1,\dots,e_n\}$ be the set of generators of $A_\Gamma^+$. We say $\{T_j\}$ satisfies property (P) if for there exists $0\leq \rho<1$ such that for all $\rho<r<1$, 
\[\Delta_{r,T}:=\sum_{U\subset F} (-1)^{|U|} r^{2|U|} T_{\vee U} T_{\vee U}^* \geq 0\]
\end{definition}

\begin{lemma}\label{lm.standard.tech} Let $\{T_j\}$ on $\bh{H}$ be a $\Gamma$-family that satisfies the weak Brehmer's condition and the property (P). Then for each $\rho<r<1$, 
    \[\sum_{p\in A_\Gamma^+} r^{2|p|}T_p \Delta_{r,T} T_p^* = I.\]
\end{lemma}
\begin{proof}
Set $F=\{e_1,\dots,e_n\}$ and expand the left hand side:
\begin{align*}
    \sum_{p\in A_\Gamma^+} r^{2|p|}T_p \Delta_{r,T} T_p^* &= \sum_{m=0}^\infty \sum_{|p|=m} \sum_{U\subset F} (-1)^{|U|}r^{2(m+|U|)} T_{p\cdot \vee U} T_{p\cdot \vee U}^* \\
    &= \sum_{m=0}^\infty \sum_{|q|=m} \left(\sum_{(p,U), p\cdot \vee U=q} (-1)^{|U|}\right) r^{2|q|} T_q T_q^*
\end{align*}
Now for a fixed $q\in A_\Gamma^+$, $q=p\cdot \vee U$ implies that each element in $U$ comes from a final syllable of $q$. Let $F(q)=\{e_i: i\text{ is a final vertex of }q\}$. Then we have
\[\sum_{(p,U), p\cdot \vee U=q} (-1)^{|U|} =\sum_{U\subset F(q)} (-1)^{|U|}.\]
It is a standard combinatorial fact that this sum is $0$ unless $F(q)=\emptyset$, which corresponds to $q=1$. Therefore, the only non-zero term in the summation is $T_1T_1^*=I$. 
\end{proof}

\begin{lemma}\label{lm.PoissonKernel} Let $\{T_j\}$ on $\bh{H}$ be a $\Gamma$-family that satisfies the weak Brehmer's condition and the property (P). Define $K_{r,T}:\cH\to\ell^2(A_\Gamma^+)\otimes\cH$ by
\[K_{r,T} h := \left(I\otimes \Delta_{r,T}^{1/2}\right) C_{r,T} h.\]
Then $K_{r,T}$ is an isometry for each $\rho< r<1$. 
\end{lemma}

\begin{proof} From definition and apply Lemma~\ref{lm.standard.tech},
\[\|K_{r,T}h\|^2 = \left\|\sum_{p\in A_\Gamma^*} e_p \otimes r^{|p|}\Delta_{r,T}^{1/2} T_p^* h\right\|^2= \left\langle \sum_{p\in A_\Gamma^+} r^{2|p|}T_p \Delta_{r,T} T_p^* h, h\right\rangle=\|h\|^2. \qedhere\]
\end{proof}

The family $\{K_{r,T}\}$ is called the \emph{Poisson kernel}. 
Now define $P_{r,T}: C^*_\lambda(A_\Gamma^+) \to \bh{H}$ by
\[P_{r,T}(a) := (K_{r,T})^* (a\otimes I_\cH) K_{r,T}.\]
\begin{proposition}\label{prop.Pr} Let $\{T_j\}$ on $\bh{H}$ be a $\Gamma$-family that satisfies the weak Brehmer's condition and the property (P).
\begin{enumerate}
    \item For each $\rho<r<1$ and $p,q\in A_\Gamma^+$
    \[ r^{|p|+|q|} T_p T_q^* = P_{r,T}(\lambda_p \lambda_q^*).\]
    \item $P_{r,T}$ is unital completely contractive linear map.
\end{enumerate}
\end{proposition}

\begin{proof} Take $p,q\in A_\Gamma^+$ and $h,k\in \cH$, 
\begin{align*}
    \langle P_{r,T}(\lambda_p \lambda_q^*)h, k\rangle &= \left\langle (\lambda_p \lambda_q^*\otimes I_\cH) \sum_{s\in A_\Gamma^+} e_s\otimes r^{|s|} \Delta_{r,T}^{1/2} T_s^* h, \sum_{t\in A_\Gamma^+} e_t\otimes r^{|t|} \Delta_{r,T}^{1/2} T_t^* h\right\rangle \\
    &=  \left\langle \sum_{s=qq'\in A_\Gamma^+} e_{q'}\otimes r^{|q|+|q'|} \Delta_{r,T}^{1/2} T_{qq'}^* h, \sum_{t=pp'\in A_\Gamma^+} e_{p'}\otimes r^{|p|+|p'|} \Delta_{r,T}^{1/2} T_{pp'}^* h\right\rangle \\
    &=  \left\langle r^{|p|+|q|} T_p \left(\sum_{p'=q'\in A_\Gamma^+} r^{2|p'|} T_{p'} \Delta_{r,T} T_{p'}^* \right) T_q^* h, k\right\rangle \\
    &= \left\langle r^{|p|+|q|} T_p T_q^* h, k\right\rangle
\end{align*}
Here, the last equality follows from Lemma~\ref{lm.standard.tech}. 

By Lemma~\ref{lm.PoissonKernel}, $K_{r,T}$ is an isometry and thus $P_{r,T}$ is completely contractive and linear. Take $p=q=1$, $P_{r,T}(I)=I$, so $P_{r,T}$ is also unital. 
\end{proof}

As a corollary, one has the following version of von Neumann inequality.

\begin{corollary}\label{cor.vN} Let $\{T_j\}$ on $\bh{H}$ be a $\Gamma$-family that satisfies the weak Brehmer's condition and the property (P). Let $a=\sum a_{p,q} T_p T_q^*\in\lspan\{T_pT_q^*: p,q\in A_\Gamma^+\}$. Then \[\|a\|\leq \|\sum a_{p,q} \lambda_p \lambda_q^*\|.\]
\end{corollary}

\begin{proof} By Proposition~\ref{prop.Pr}, 
\[\|\sum a_{p,q} r^{|p|+|q|} T_p T_q^*\| \leq \|\sum a_{p,q} \lambda_p \lambda_q^*\|.\]
Since this is a finite sum, taking $r\to 1$ gives the desired result.
\end{proof}

We now prove that the limit of $P_{r,T}$ as $r\to 1^-$ is well-defined on $C_\lambda^*(A_\Gamma^+)$. 

\begin{proposition}
    For each $a\in C_\lambda^*(A_\Gamma^+)$, the limit
\[\lim_{r\to 1^-} P_{r,T}(a),\]
exists in the norm topology.
\end{proposition}

\begin{proof}
Fix $\epsilon>0$.  Since $\lspan\{\lambda_p\lambda_q^*: p,q\in A_\Gamma^+\}$ is dense in $C^*_\lambda(A_\Gamma^+)$, there exists a finite sum $a'=\sum a_{p,q} \lambda_p \lambda_q^*\in\lspan\{\lambda_p\lambda_q^*: p,q\in A_\Gamma^+\}$ such that $\|a-a'\|<\epsilon/3$. Since $P_{r,T}$ is contractive,
$\|P_{r,T}(a)-P_{r,T}(a')\|<\|a-a'\|<\epsilon/3$ for all $0\leq r<1$. Consider $P_{r,T}(a')=\sum a_{p,q} r^{|p|+|q|} T_p T_q^*$. Since this is a finite sum, $\lim_{r\to 1^-} P_{r,T}(a')$ exists. Therefore, there exists $r_0$ such that for all $r_0<r_1,r_2<1$, $\|P_{r_1,T}(a')-P_{r_2,T}(a')\|<\epsilon/3$. 
Therefore, for  all $r_0<r_1,r_2<1$,
\[
    \|P_{r_1,T}(a)-P_{r_2,T}(a)\| \leq \|P_{r_1,T}(a)-P_{r_1,T}(a')\| + \|P_{r_1,T}(a')-P_{r_2,T}(a')\| + \|P_{r_2,T}(a)-P_{r_2,T}(a')\| <\epsilon.
\]
This implies $\lim_{r\to 1^-} P_{r,T}(a)$ exists in the norm topology.
\end{proof}
The \emph{Poisson transformation} is defined as $\Phi_T:C_\lambda^*(A_\Gamma^+) \to \bh{H}$ by
\[\Phi_T(a) := \lim_{r\to 1^-} P_{r,T}(a).\]

\begin{theorem}\label{thm.poisson.main} Let $\{T_j\}$ on $\bh{H}$ be a $\Gamma$-family that satisfies the weak Brehmer's condition and the property (P). Then the Poisson transform $\Phi_T:C_\lambda^*(A_\Gamma^+) \to \bh{H}$ satisfies the following properties.
\begin{enumerate}[label=\textup{(\arabic*)}]
    \item\label{item.Poisson1} For each $p,q\in A_\Gamma^+$, $\Phi_T(\lambda_p\lambda_q^*)=T_pT_q^*$.
    \item\label{item.Poisson2} $\Phi_T$ is a unital completely contractive map.
    \item\label{item.Poisson3} $\Phi_T(\lambda_p)\Phi_T(\lambda_q)=\Phi_T(\lambda_{pq})$ for all $p,q\in A_\Gamma^+$. 
\end{enumerate}
\end{theorem}

\begin{proof} \ref{item.Poisson1} and \ref{item.Poisson2} follow immediately from Proposition~\ref{prop.Pr} and \ref{item.Poisson3} follows from \ref{item.Poisson1}.
\end{proof}

\begin{corollary}\label{cor.star.regular} Let $\{T_j\}$ on $\bh{H}$ be a $\Gamma$-family. Then $\{T_j\}$ has a $\ast$-regular dilation if and only if it satisfies the weak Brehmer's condition and the property (P). 
\end{corollary}

\begin{proof}
For $\Rightarrow$: it was shown in \cite{Li2017} that having $\ast$-regular dilation implies property (P). It is also equivalent to Brehmer's condition (Condition~\ref{cond.Brehmer.Artin}), which is stronger than the weak Brehmer's condition. 

For $\Leftarrow$: let $\Phi_T:C_\lambda^*(A_\Gamma^+)\to\bh{H}$ be the Poisson transform, which is a unital completely positive map as proved in Theorem~\ref{thm.poisson.main}. The Stinespring's dilation theorem states that $\Phi_T$ must be a compression of a $*$-homomorphism. In other words, there exists a $*$-homomorphism $\pi: C^*_\lambda(A_\Gamma^+)\to\bh{K}$ for some Hilbert space $\cK\supset\cH$ such that for all $a\in C^*_\lambda(A_\Gamma^+)$,
\[\Phi_T(a) = P_\cH \pi(a)|_\cH.\]
Let $V_p=\pi(\lambda_p)\in \bh{K}$. Since $\pi$ is a $*$-homomorphism, $V:A_\Gamma^+\to\bh{K}$ by $V(p)=V_p$ is an isometric representation. Moreover,
for each $p,q\in P$,
\[P_\cH V_p^* V_q|_H = \Phi_T(\lambda_p^* \lambda_q) = \begin{cases}
    \Phi_T(\lambda_{p^{-1}r} \lambda_{q^{-1}r}^*), &\text{if } p\vee q=r,\\
    0, &\text{otherwise}. 
    \end{cases}= \begin{cases}
    T_{p^{-1}r} T_{q^{-1}r}^*, &\text{if } p\vee q=r,\\
    0, &\text{otherwise}. 
    \end{cases} \]
Therefore, $V$ is a $*$-regular dilation of $T$.
\end{proof}


\end{document}